\pdfoutput=1
\RequirePackage{ifpdf}
\ifpdf 
\documentclass[pdftex]{sigma}
\else
\documentclass{sigma}
\fi

\newtheorem{Theorem}{Theorem}[section]
\newtheorem{Observation}[Theorem]{Observation}
 { \theoremstyle{definition}
\newtheorem{Remark}[Theorem]{Remark} }

\begin{document}

\allowdisplaybreaks

\newcommand{\arXivNumber}{1612.03674}

\renewcommand{\PaperNumber}{029}

\FirstPageHeading

\ShortArticleName{Isomonodromy for the Degenerate Fifth Painlev\'e Equation}

\ArticleName{Isomonodromy for the Degenerate\\ Fifth Painlev\'e Equation}

\Author{Primitivo B.~ACOSTA-HUM\'ANEZ~$^\dag$, Marius VAN DER PUT~$^\ddag$ and Jaap TOP~$^\ddag$}

\AuthorNameForHeading{P.B.~Acosta-Hum\'anez, M.~van der Put and J.~Top}

\Address{$^\dag$~Universidad Sim\'{o}n Bol\'{\i}var, Barranquilla, Colombia}
\EmailD{\href{mailto:primitivo.acosta@unisimonbolivar.edu.co}{primitivo.acosta@unisimonbolivar.edu.co}}

\Address{$\ddag$~University of Groningen, Groningen, The Netherlands}
\EmailD{\href{mailto:m.van.der.put@rug.nl}{m.van.der.put@rug.nl}, \href{mailto:j.top@rug.nl}{j.top@rug.nl}}

\ArticleDates{Received December 12, 2016, in f\/inal form May 01, 2017; Published online May 09, 2017}

\Abstract{This is a sequel to papers by the last two authors making the Riemann--Hilbert correspondence and isomonodromy explicit. For the degenerate f\/ifth Painlev\'e equation, the moduli spaces for connections and for monodromy are explicitly computed. It is proven that the extended Riemann--Hilbert morphism is an isomorphism. As a consequence these equations have the Painlev\'e property and the Okamoto--Painlev\'e space is identif\/ied with a~moduli space of connections. Using MAPLE computations, one obtains formulas for the degenerate f\/ifth Painlev\'e equation, for the B\"acklund transformations.}


\Classification{33E17; 14D20; 14D22; 34M55}

\section{Introduction}
In the series of papers \cite{vdP2,vdP1, vdP-Sa,vdP-T1, vdP-T2} on isomonodromy families for Painlev\'e equations the cases ${\rm P}_{\rm I}$--${\rm P}_{\rm IV}$ are treated. Here we apply our methods to ${\rm degP}_{\rm V}$, the degenerate f\/ifth Painlev\'e equation. We hope to extend this in a~later paper to ${\rm P}_{\rm V}$. We now describe the method of the Riemann--Hilbert correspondence for ${\rm degP}_{\rm V}$, following closely~\cite{vdP-T1,vdP-T2}.

The degenerate f\/ifth Painlev\'e equation ${\rm degP}_{\rm V}(\theta_0,\theta_1)$ depends on two parameters $\theta_0$, $\theta_1$ and the corresponding isomonodromy family is given, to begin with, by the {\it set} ${\bf S}(\theta_0,\theta_1)$ of dif\/ferential modules $M$ over $\mathbb{C}(z)$ def\/ined by: $\dim M=2$; the exterior product $\Lambda ^2M$ is trivial; the points~$0$,~$1$ are regular singular with local exponents $\pm \frac{\theta_0}{2}$ and $\pm \frac{\theta_1}{2}$. Finally $z=\infty$ is irregular singular with Katz invariant $\frac{1}{2}$, which means that the `generalized eigenvalues' at $z=\infty$ are $\pm t\cdot z^{1/2}$ with $t\in \mathbb{C}^*$.

This set is made into an algebraic variety $\mathcal{M}(\theta_0,\theta_1)$ which is a~moduli space for connections on a f\/ixed bundle on $\mathbb{P}^1$ of rank two and degree $-1$ with prescribed data (see Section~\ref{11}) at $z=0,1,\infty$. If $\theta_0\neq 0$, $\theta_1\neq 0$, then $\mathcal{M}(\theta_0,\theta_1)$ is a f\/ine modul space and is smooth. For $\theta_0=0$ and/or $\theta_1=0$, the moduli problem will be changed by adding ``an invariant line''. This is called a~parabolic structure in the literature, see \cite{Ina06, IISA,IIS1,IIS2}. It leads to a~f\/ine moduli space $\mathcal{M}^+(\theta_0, \theta_1)$ which is a~desingularisation of $\mathcal{M}(\theta_0,\theta_1)$.

The analytic data attached to modules in ${\bf S}(\theta_0,\theta_1)$ are two monodromy matrices and one Stokes matrix. They produce a `monodromy space' $\mathcal{R}(s_0,s_1)$ depending on $s_0=e^{\pi i \theta_0}+e^{-\pi i \theta_0}$ and $s_1=e^{\pi i \theta_1}+e^{-\pi i \theta_1}$. For $s_0\neq \pm 2$ and $s_1\neq \pm 2$ the monodromy space is a f\/ine moduli space and is smooth. It is in fact a smooth af\/f\/ine cubic surface with three lines at inf\/inity. For the other cases one changes the moduli problem by adding an `invariant line'. The f\/ine moduli space $\mathcal{R}^+(s_0,s_1)$ for these new data is a minimal resolution of $\mathcal{R}(s_0,s_1)$.

The Stokes matrix attached to a module in $ {\bf S}(\theta_0, \theta_1)$ depends on the choice of a~summation direction at $z=\infty$. This direction has to be dif\/ferent from the singular direction at $z=\infty$ which turns around $\infty$ for~$t$ varying in $T:=\mathbb{C}^*$. Therefore there is a locally, with respect to $t\in \mathbb{C}^*$, def\/ined analytic Riemann--Hilbert map ${\rm RH}\colon \mathcal{M}^+(\theta_0,\theta_1) \rightarrow \mathcal{R}^+(s_0,s_1)$ with $s_0=e^{\pi i \theta_0}+e^{-\pi i \theta_0}$ and $s_1=e^{\pi i \theta_1}+e^{-\pi i \theta_1}$. After replacing $T$ by its universal covering $\tilde{T}=\mathbb{C}$ the morphism ${\rm RH}$ is well def\/ined. {\it The main result is that the extended Riemann--Hilbert map
\begin{gather*}
{\rm RH}^+\colon \ \mathcal{M}^+(\theta_0,\theta_1)\times _T\tilde{T}\rightarrow \mathcal{R}^+(s_0,s_1)\times \tilde{T}
\end{gather*}
is an analytic isomorphism}. It follows from this that ${\rm degP}_{\rm V}$ has the Painlev\'e property and that $\mathcal{M}^+(\theta_0,\theta_1)$ coincides with the Okamoto--Painlev\'e space (see~\cite{S-Ta,STT02,STe02} for this subject). The explicit computations of the spaces $\mathcal{M}^+(\theta_0,\theta_1)$ and $\mathcal{R}^+(s_0,s_1)$ lead, using MAPLE, to formulas for ${\rm degP}_{\rm V}$ and for the B\"acklund transformations.

We note that our def\/inition of ${\rm degP}_{\rm V}$ is not quite the same as the `classical' degenerate f\/ifth Painlev\'e equation. The close relation between the two is given in Section~\ref{section3}. We were informed by Y.~Ohyama about the `equivalence' between ${\rm degP}_{\rm V}$ and ${\rm P}_{\rm III}(D_6)$ found by V.I.~Gromak~\cite{Gr1}. A~Hamiltonian for ${\rm P}_{\rm III}(D_6)$ is
\begin{gather*} \frac{1}{t}\big(q^2 p^2-\big(q^2-(\alpha+\beta) q-t\big)p - \alpha q\big).\end{gather*}
The degrees in $p$ and $q$ are $\leq 2$. Therefore the Hamilton equations allow to eliminate~$p$ (in terms of $q$, $q'$) and also to eliminate~$q$ (in
terms of~$p$,~$p'$). In the f\/irst case one obtains ${\rm P}_{\rm III}(D_6)$ and in the second case the classical degenerate f\/ifth Painlev\'e equation.

This equivalence does not seem to produce a relation between the isomonodromy families for ${\rm degP}_{\rm V}$ and for ${\rm P}_{\rm III}(D_6)$.

Apart from the references given at the beginning of this introduction, several sources discuss geometric aspects of Painlev\'{e} equations. Many of these can be found in the references of our papers \cite{vdP2,vdP1,vdP-Sa,vdP-T1, vdP-T2}. Classical' papers on the subject are \cite{JMU,JM,O1,O2,O3,O4,O5}. Especially relevant for the present text are the paper by Ohyama and Okumura \cite{OO2}, Witte's paper \cite{Wi}. The book \cite{FIKN} by Fokas, Its, Kapaev, and Novokshenov discusses more analytic aspects of the Riemann--Hilbert correspondence, but it does not discuss the degenerate f\/ifth Painlev\'e equation. Finally, we mention the recent paper \cite{CMR} by Chekhov, Mazzocco, and Rubtsov which also provides an interesting geometric appoach and overview.

\section[The moduli space $\mathcal{M}(\theta_0,\theta_1)$ of connections]{The moduli space $\boldsymbol{\mathcal{M}(\theta_0,\theta_1)}$ of connections}\label{section2}
\subsection[Def\/inition of ${\bf S}(\theta_0,\theta_1)$]{Def\/inition of $\boldsymbol{{\bf S}(\theta_0,\theta_1)}$}\label{11}
An element of the set ${\bf S}(\theta_0,\theta_1)$ is (the isomorphy class of) a tuple $(M,\theta_0,\theta_1, t)$ where $M$ is a~dif\/ferential module over $\mathbb{C}(z)$ such that $\dim M=2$; $\det M:=\Lambda ^2M$ is the trivial module; $M$~has three singular points $0$, $1$, $\infty$ and their Katz invariants are $r(0)=0$, $r(1)=0$, $r(\infty)=1/2$. Further, the singularities are represented:
\begin{itemize}\itemsep=0pt
\item at $z=0$ by $\frac{d}{dz}+\frac{1}{z}\left(\begin{smallmatrix} \omega _0& * \\ 0& -\omega _0 \end{smallmatrix}\right)$ with
$\omega _0=\frac{\theta_0}{2}$ and $\theta _0, *\in \mathbb{C}$;
\item at $z=1$ by $\frac{d}{dz}+\frac{1}{z-1}\left(\begin{smallmatrix} \omega _1& * \\ 0& -\omega _1 \end{smallmatrix}\right)$ with
$\omega _1=\frac{\theta_1}{2}$ and $\theta_1, *\in \mathbb{C}$;
\item at $z=\infty$ by $z\frac{d}{dz}+\left(\begin{smallmatrix} \omega _\infty & 0 \\ 0 & -\omega _\infty \end{smallmatrix}\right)$ with
$\omega _\infty=tz^{1/2}$ and $t\in \mathbb{C}^*$.
\end{itemize}

Let $\partial$ denote the given dif\/ferential operator on $M$, corresponding with the derivation $\frac{d}{dz}$ on~$\mathbb{C}(z)$, i.e., $\partial(fm) =\frac{df}{dz}m+f\partial (m)$ for $f\in \mathbb{C}(z)$ and $m\in M$.

The condition at $z=0$ means that $\mathbb{C}((z))\otimes M$ has a~$\mathbb{C}[[z]]$-lattice with a~basis such that the matrix of $z\partial$ is $\left(\begin{smallmatrix} \omega_0& *\\ 0 & -\omega_0 \end{smallmatrix}\right)$. For $\theta_0\neq 0$ one can always take $*=0$. For $\theta_0=0$, any~$*$
is admissible. At $z=1$ the condition is similar for the matrix of $(z-1)\partial$.

The condition at $z=\infty$ means that $\mathbb{C}((z^{-1/2}))\otimes M$ has a basis for which $z\partial$ has the matrix $\left(\begin{smallmatrix} \omega _\infty & 0 \\ 0 & -\omega_\infty \end{smallmatrix}\right)$. In the construction of the moduli space $\mathcal{M}(\theta_0,\theta _1)$ this matrix will be changed in a~matrix def\/ined over $\mathbb{C}((z^{-1}))$.

There is an obvious bijection ${\bf S}(\theta_0,\theta_1)\rightarrow {\bf S}(\theta_0+a,\theta _1+b)$ for any $a,b\in 2\mathbb{Z}$, obtained by changing the lattices for $z=0$ and $z=1$.

\subsection{Choosing connections}
We represent $M\in {\bf S}(\theta_0,\theta_1)$ by a connection $\nabla \colon \mathcal{V}\rightarrow \Omega ([0]+[1]+2[\infty])\otimes\mathcal{V}$, where $\mathcal{V}$ is a vector bundle on $\mathbb{P}^1$ of rank two. Thus the generic f\/iber of $\nabla$ is the map $M\rightarrow M\otimes _{\mathbb{C}(z)}\mathbb{C}(z)dz$, given by $m\mapsto \partial(m)dz$.

The vector bundle $\mathcal{V}$ is determined by the choice of the lattices at every point of $\mathbb{P}^1$. For the points $z=a\neq 0,1,\infty$ the lattice in $\mathbb{C}((z-a))\otimes M$ is the $\mathbb{C}[[z-a]]$-module generated by $\ker (\partial ,\mathbb{C}((z-a))\otimes M)$. For the points $z=0$ and $z=1$ we choose the lattices corresponding to the given $\theta_0$, $\theta_1$. In other words, the characteristic polynomials of the matrices of $z\partial$ and $(z-1)\partial$ are prescribed by $X^2-\frac{\theta _0^2}{4}$ and $X^2-\frac{\theta _1^2}{4}$.

At $z=\infty$ the situation is more complicated. An invariant lattice at $\infty$ is represented over the dif\/ferential f\/ield $\mathbb{C}((z^{-1/2}))$ by
$z\frac{d}{dz}+\left(\begin{smallmatrix} tz^{1/2}&0\\ 0&-tz^{1/2} \end{smallmatrix}\right). $ We need an expression over the f\/ield $\mathbb{C}((z^{-1}))$ or ``invariant lattices'' for $\mathbb{C}((z^{-1}))\otimes M$. For Katz invariant $1/2$ or~$1$ at~$\infty$, a~lattice~$\Lambda$ is called invariant if $z\partial \Lambda \subset \Lambda$.

We adopt here the terminology \cite{vdP-Si} for the classif\/ication of dif\/ferential modules over $\mathbb{C}((z^{-1}))$.

The formal solution space $V$ at $z=\infty$ is described as $V=V_q\oplus V_{-q}=\mathbb{C}e_1\oplus \mathbb{C}e_2$ with $q=tz^{1/2}$ and the formal monodromy $\gamma $ is given by $\gamma (e_1)=e_2$ and $\gamma (e_2)=-e_1$. Indeed, the determinant of $M$ is trivial and hence $\det(\gamma)=1$. The (formal local) dif\/ferential module $\mathbb{C}((z^{-1}))\otimes M$ and its invariant lattices are now obtained by considering the invariants of~$\mathcal{U}\otimes_{\mathbb{C}} V$ (here $\mathcal{U}$ is the universal Picard--Vessiot ring for $\mathbb{C}((z^{-1}))$) under the actions of the dif\/ferential automorphisms of $\mathcal{U}$ over $\mathbb{C}((z^{-1}))$. A computation yields invariant lattices~$\Lambda_1$ and~$\Lambda _2$ represented by
\begin{gather*} z\frac{d}{dz}+\left(\begin{matrix} -\frac{1}{4}&tz\\
t&\frac{1}{4} \end{matrix}\right) \qquad \text{and}\qquad z\frac{d}{dz}+\left(\begin{matrix} -\frac{3}{4}&tz\\
t&-\frac{1}{4} \end{matrix}\right).
\end{gather*}
All lattices are given by $z^n\Lambda _1$ and $z^n\Lambda _2$ with $n\in \mathbb{Z}$.

\begin{Remark} By conjugation with the constant matrix $\left(\begin{smallmatrix} t & 0\\ 0 & 1 \end{smallmatrix}\right)$ one can
change these formulas into
\begin{gather*} z\frac{d}{dz}+\left(\begin{matrix} -\frac{1}{4}&t^2z\\
1&\frac{1}{4} \end{matrix}\right) \qquad \text{and}\qquad
z\frac{d}{dz}+\left(\begin{matrix} -\frac{3}{4}&t^2z\\
1&-\frac{1}{4} \end{matrix}\right).
\end{gather*}
\end{Remark}

We want that $\det \mathcal{V}:=\Lambda ^2\mathcal{V}$ has degree~$-1$. The reason is that, in general $\mathcal{V}=O(d_1)\oplus O(d_2)$ with $d_1\leq d_2$. The dif\/ferential module $\mathbb{C}((z^{-1}))\otimes M$ is irreducible because of the ramif\/ication at $z=\infty$. Hence~$M$ and also the connection on $\mathcal{V}$ is irreducible. This implies (in this special case) that $d_2-d_1\leq 2$. If we make the natural assumption that $\deg \mathcal{V}=0$, then there are two possibilities for~$\mathcal{V}$, namely $O\oplus O$ and $O(-1)\oplus O(+1)$. The construction of a~family of connection would involve a construction of a family of vector bundles of rank two and degree~0 on~$\mathbb{P}^1$. However, we will avoid this by imposing $\deg \mathcal{V}=-1$. Then there is only one possibility for $\mathcal{V}$, namely $O\oplus O(-1)$.

Finally, the degree of $\mathcal{V}$ is $-1$ precisely for the choice of the lattice $\Lambda _2$ corresponding to $z\frac{d}{dz}+\left(\begin{smallmatrix} -\frac{3}{4}&tz\\ t&-\frac{1}{4} \end{smallmatrix}\right)$.

\subsection{Computation of the connection}\label{section2.3}
We identify $\mathcal{V}$ with $Oe_1\oplus O(-[\infty ])e_2$ (a subvector bundle of the free (i.e., trivial) vector bundle $Oe_1\oplus Oe_2$). Let $D=\nabla _{\frac{d}{dz}}$ and $\tilde{D}:=z(z-1)D$. Then $\tilde{D}e_1 \in \langle 1,z,z^2\rangle e_1+\langle 1,z\rangle e_2$; $\tilde{D}e_2 \in \langle 1,z,z^2,z^3\rangle e_1+\langle 1,z,z^2\rangle e_2$. Here $\langle *,* ,* \rangle$ denotes the $\mathbb{C}$-vector spaces generated by these $*$'s. The vector bundle $\mathcal{V}$ has an automorphism group $G$ and the moduli space $\mathcal{M}(\theta_0,\theta _1)$ (note that we f\/ix the parameters $\theta_0$ and $\theta _1$ in this construction), that we are constructing, is obtained by dividing the space of all matrices for the connections by the group $G$. This group consists of the elements $e_1\mapsto \lambda e_1$, $e_2\mapsto \mu e_2+(x_0+x_1z)e_1$ (with $\lambda , \mu \in \mathbb{C}^*$, $x_0,x_1\in \mathbb{C}$). Further the multiples of the identity act trivially and we have only to consider the automorphisms $e_1\mapsto \lambda e_1$, $e_2 \mapsto e_2+(x_0+x_1z)e_1$.

Since the connection is irreducible we have that $\tilde{D}e_1=\langle 1,z,z^2\rangle e_1+(b_1z+b_0)e_2$ with $(b_1z+b_0)\neq 0$. We consider now two af\/f\/ine parts: $b_1\neq 0$ and $b_0\neq 0$. The two af\/f\/ine parts are divided out by the action of~$G$. These quotients are geometric quotients and they are obtained by normalization of some of the entries of the matrices. One obtains two af\/f\/ine varieties $\mathcal{M}_1(\theta _0,\theta _1)$ and
$\mathcal{M}_2(\theta _0,\theta_1)$ which are glued to the moduli space $\mathcal{M}(\theta_0,\theta_1)$.

\subsubsection[The f\/irst af\/f\/ine part $\mathcal{M}_1(\theta_0, \theta_1)$]{The f\/irst af\/f\/ine part $\boldsymbol{\mathcal{M}_1(\theta_0, \theta_1)}$}\label{section2.3.1}

This is obtained by dividing the open subspace $b_1\neq 0$ by the action of the group~$G$. First one normalizes $b_1$ by the automorphism $e_1\mapsto \lambda e_1$ to $b_1=1$. Using $e_2\mapsto e_2+(a_0+a_1z)e_1$ one normalizes further to $\tilde{D}e_1=a_0e_1+(z+b_0)e_2$ and $\tilde{D}e_2= (c_0+c_1z+c_2z^2+c_3z^3)e_1+(d_0+d_1z+d_2z^2)e_2$. Now we have to compute the equations between the variables due to the prescription of the three invariant lattices.

For $z=0$. Now $\frac{1}{z-1}\tilde{D}$ has modulo $z=0$ the eigenvalues $\pm \frac{\theta_0}{2}$. This yields $a_0+d_0=0$ and $a_0d_0-b_0c_0=-\frac{\theta_0^2}{4}$.

For $z=1$. Now $\frac{1}{z}\tilde{D}$ has modulo $z=1$ the eigenvalues $\pm \frac{\theta_1}{2}$. This yields $ a_0+d_0+d_1+d_2=0$ and $a_0(d_0+d_1+d_2)-(1+b_0)(c_0+c_1+c_2+c_3)=-\frac{\theta_1^2}{4}.$

For $z=\infty$. The local basis of $\mathcal{V}$ at $z=\infty$ is $e_1$, $z^{-1}e_2$. The operator $\frac{1}{z-1}\tilde{D}$ w.r.t.\ this local basis is $z\frac{d}{dz}+\left(\begin{smallmatrix} \frac{a_0}{z-1}&\frac{c_0+c_1z+c_2z^2+c_3z^3}{z(z-1)}\\ z\frac{z+b_0}{z-1} & -1+\frac{d_0+d_1z+d_2z^2}{z-1} \end{smallmatrix}\right)$. This should be equivalent to $z\frac{d}{dz}+\left(\begin{smallmatrix}\frac{-3}{4}&tz\\ t&\frac{-1}{4}\end{smallmatrix}\right)$. The characteristic polynomials of the two matrices should be equal modulo $\mathbb{C}[[z^{-1}]]$.

\begin{Remark}
A more precise computation is needed to verify the correctness of the last statement. On another basis of the lattice, the operator $z\frac{d}{dz}+\left(\begin{smallmatrix}\frac{-3}{4}&tz\\ t&\frac{-1}{4}\end{smallmatrix}\right)$ reads
\begin{gather*}
\big(1+A_1z^{-1}+\cdots \big)^{-1}A_0^{-1}\left(z\frac{d}{dz}+\left(\begin{matrix}\frac{-3}{4}&tz\\ t&\frac{-1}{4}\end{matrix}\right)\right)A_0\big(1+A_1z^{-1}+\cdots\big),
\end{gather*} leading to
\begin{gather*}
z\frac{d}{dz}+A_0^{-1}\left(\left(\begin{matrix}\frac{-3}{4}&tz\\ t& \frac{-1}{4}\end{matrix}\right) + t\left( \begin{matrix} x_3& x_4-x_1\\ 0 & -x_3\end{matrix}\right) \right)A_0,
\end{gather*}
where $A_0\in {\rm GL}_2(\mathbb{C})$, $x_1,x_3,x_4\in \mathbb{C}$. Thus $d_2=0$, $d_1=0$, $c_3=0$ and $c_2=t^2$.
\end{Remark}

{\it Finally} the connection as matrix dif\/ferential operator w.r.t.\ $e_1$, $e_2$ reads
\begin{gather*}
\frac{d}{dz}+\frac{1}{z(z-1)}\left(\begin{matrix}a_0&c_0+c_1z+t^2z^2\\ z+b_0&-a_0\end{matrix}\right)
\end{gather*}
and there are two equations
\begin{gather*}
a_0^2+b_0c_0=\frac{\theta_0^2}{4},\qquad a_0^2+(1+b_0)\big(c_0+c_1+t^2\big)=\frac{\theta_1^2}{4}.\end{gather*}
The last equation can be changed into
\begin{gather*}
c_0=-\frac{\theta_0^2}{4}+\frac{\theta_1^2}{4}-c_1-t^2-b_0c_1-b_0t^2\end{gather*}
and this is used to eliminate $c_0$. This leaves four variables $a_0$, $b_0$, $c_1$, $t$ and one equation
\begin{gather*}
a_0^2+b_0\left(-\frac{\theta_0^2}{4}+\frac{\theta_1^2}{4}-c_1-t^2-b_0c_1-b_0t^2\right)- \frac{\theta _0^2}{4}=0.
\end{gather*}
{\it The above describes the space $\mathcal{M}_1(\theta_0,\theta_1)$}. We note that this space depends only on $\theta_0^2$ and $\theta_1^2$. The space $\mathcal{M}_1(\theta_0,\theta_1)$ is smooth if $\theta_0\neq 0$ and $\theta_1\neq 0$. For $\theta_0=0$, $\theta_1\neq 0$, the {\it singular locus} is given by $a_0=0$, $b_0=0$, $c_1=-t^2+\frac{\theta_1^2}{4}$. For $\theta_0\neq 0$, $\theta_1=0$, the {\it singular locus} is given by $a_0=0$, $b_0=-1$, $c_1=-t^2+\frac{\theta_0^2}{4}$. For $\theta_0=\theta_1=0$, the singular locus is given by $a_0=0$, $b_0(b_0+1)=0$, $c_1=-t^2$.

Moreover $\mathcal{M}_1(\theta_0,\theta_1)$ seen as a two dimensional space over the f\/ield $\mathbb{C}(t)$ has the same singular locus but now seen as a set of at most two points.

\begin{Observation}\label{Observation1.1} For every $a\in \mathbb{C}^*$ the closed subspace of $\mathcal{M}_1(\theta_0,\theta_1)$, defined by $t=a$, is simply connected.
\end{Observation}

Indeed, the equation
\begin{gather*}
a_0^2+b_0\left(-\frac{\theta_0^2}{4}+\frac{\theta_1^2}{4}-c_1-a^2-b_0c_1-b_0a^2\right)- \frac{\theta _0^2}{4}=0
\end{gather*}
def\/ines this two-dimensional space. It is mapped to $\mathbb{C}^2$ by $(a_0,b_0,c_1)\mapsto (a_0,b_0)$. The f\/ibre is either empty, or a point, or $\mathbb{C}$. Hence it suf\/f\/ices to show that the image $B\subset \mathbb{C}^2$ is simply connected. Now $B$ is the union of $X:=\mathbb{C}\times (\mathbb{C}\setminus \{0,-1\})$ and the points $(\pm \frac{\theta_0}{2},0)$ and $(\pm \frac{\theta_1}{2},-1)$. The canonical map $\pi_1(X,*)\rightarrow \pi_1(B,*)$ is surjective. Consider one of the two generators, $s\in [0,1]\mapsto \big(\frac{\theta_0}{2},e^{2\pi i s}\big)$ of $\pi_1(X,*)$. In $B$ this closed path is homotopic to the constant closed path by the homotopy $s,\lambda \in [0,1]\mapsto \big(\frac{\theta_0}{2},\lambda e^{2\pi i s}\big)$. The same observation can be made for the other generator of $\pi_1(X,*)$. Hence $B$ is simply connected.

\subsubsection[The second af\/f\/ine part $\mathcal{M}_2(\theta_0,\theta_1)$]{The second af\/f\/ine part $\boldsymbol{\mathcal{M}_2(\theta_0,\theta_1)}$}\label{section2.3.2}
This space is obtained by dividing the open subset $b_0\neq 0$ by the action of $G$. Now $b_0\neq 0$ is normalized by the automorphism $e_1\mapsto \lambda e_1$ to $b_0=1$.

Using $e_2\mapsto e_2+(x_0+x_1z)e_2$ one normalizes further to $\tilde{D}e_1=a_2z^2e_1+(b_1z+1)e_2$ and $\tilde{D}e_2=(c_0+c_1z+c_2z^2+c_3z^3)e_1+(d_0+d_1z+d_2z^2)e_2$. The equations derived from the prescribed invariant lattices are

For $z=0$: $d_0=0$ and $c_0=\frac{\theta_0^2}{4}$.

For $z=1$: $a_2+d_1+d_2=0$ and $a_2(d_1+d_2)-(1+b_1)(c_0+c_1+c_2+c_3)=-\frac{\theta_1^2}{4}$.

For $z=\infty$: the operator $\frac{1}{z-1}\tilde{D}$ w.r.t.\ $e_1$, $z^{-1}e_2$ is
\begin{gather*}
z\frac{d}{dz}+\left(\begin{matrix} \frac{a_2z^2}{z-1}&
\frac{c_0+c_1z+c_2z^2+c_3z^3}{z(z-1)}\\ (1+b_1z)\frac{z}{z-1}&
-1+\frac{d_0+d_1z+d_2z^2}{z-1}\end{matrix}\right)
\end{gather*}
and is equivalent to
\begin{gather*}
z\frac{d}{dz}+\left(\begin{matrix}\frac{-3}{4}&tz\\ t &\frac{-1}{4}\end{matrix}\right).
 \end{gather*}
This yields the equations $d_2=-a_2$ (and from before $d_1+d_2=-a_2$ and thus $d_1=0$). Further $a_2^2+b_1c_3=0$ and
$c_3=t^2-b_1c_2-a_2$ is used to eliminate $c_3$. One obtains the equation
\begin{gather*}
c_1+c_2+t^2+b_1\frac{\theta_0^2}{4}+b_1c_1-a_2+\frac{\theta_0^2}{4}-\frac{\theta_1^2}{4}=0
\end{gather*} which eliminates $c_2$. The matrix dif\/ferential operator w.r.t.\ the basis $e_1$, $e_2$ is now
\begin{gather*}
\frac{d}{dz} +\frac{1}{z(z-1)} \left(\begin{matrix} a_2z^2&\frac{\theta_0^2}{4}+c_1z+c_2z^2+(t^2-b_1c_2-a_2)z^3\\ 1+b_1z& -a_2z^2\end{matrix}\right)
\end{gather*}
in the variables $a_2$, $b_1$, $c_1$, $t$ and one equation $a_2^2+b_1(-a_2-b_1c_2+t^2)=0$ with $c_2$ eliminated as above. {\it The above describes the space $\mathcal{M}_2(\theta_0,\theta_1)$.}

The {\it singular locus} of this three-dimensional variety is given by $\theta_1=0$, $a_2=0$, $b_1=-1$, $c_1=t^2-\frac{\theta_0^2}{4}$ (note that $t\neq 0$). As a~variety over $\mathbb{C}(t)$ the singularity occurs only for $\theta_1=0$ and consists of one point.

The two parts glue to the required space $\mathcal{M}(\theta _0,\theta_1)$.
\begin{Observation}\label{Observatio1.2} For any $a\in \mathbb{C}^*$ the closed subspace of $\mathcal{M}(\theta_0,\theta_1)$, defined by $t=a$, is simply connected.
\end{Observation}

Indeed, by Observation \ref{Observation1.1}, this holds for the f\/irst open af\/f\/ine part of this space. For the second open af\/f\/ine part the same reasoning proves the statement. Van Kampen's theorem f\/inishes the proof.

\subsubsection[Resolving the singularities of $\mathcal{M}(\theta_0,\theta_1)$]{Resolving the singularities of $\boldsymbol{\mathcal{M}(\theta_0,\theta_1)}$}\label{section2.3.3}

For $\theta_0\neq 0$, $\theta_1\neq 0$, the space $\mathcal{M}(\theta _0,\theta_1)$ is the {\it fine moduli space} for the connections on the f\/ixed vector bundle $\mathcal{V}$ of rank 2 and degree $-1$ with the prescribed singularities. The space $\mathcal{M}(\theta_0,\theta_1)$ is smooth and simply connected for any nonzero f\/ixed value of~$t$. Its set of closed points $\mathcal{M}(\theta_0,\theta_1)(\mathbb{C})$ coincides by construction with ${\bf S}(\theta_0,\theta_1)$.

For $\theta_0=0$ and/or for $\theta_1=0$ the space $\mathcal{M}(\theta_0,\theta_1)$ has singularities and is no longer a f\/ine moduli space. Consider the case $\theta_0=0$ and $\theta_1\neq 0$.

A tuple $(M,\theta_0=0,\theta_1,t)$ is represented at $z=0$ by a~lattice $\Lambda \subset \mathbb{C}((z))\otimes M$ such that $\delta (\Lambda )\subset \Lambda$ and the action of $\delta:=zD$ on $\Lambda /z\Lambda$ has trace zero and determinant zero. There exists a~basis of $\Lambda$ over $\mathbb{C}[[z]]$ such that $\delta$ has the matrix $\left(\begin{smallmatrix} 0& 0\\ 0 & 0 \end{smallmatrix}\right)$ or $\left(\begin{smallmatrix}0 & 1\\ 0 & 0 \end{smallmatrix}\right)$. The space $\mathcal{M}(\theta_0=0,\theta_1)$ does not distinguish between these cases. Therefore it is not a f\/ine moduli space and moreover it has a singularity.

There is a {\it geometric way} to treat these problems. One adds to the data $(M,\theta_0=0,\theta_1,t)$ a~`line'. This means the following. The assumption $\theta_0=0$ def\/ines a lattice $\Lambda\subset \mathbb{C}((z))\otimes M$ invariant under~$\delta$. The `line' is a 1-dimensional summand of $\Lambda$,
invariant under $\delta$. We note that there are two cases:
\begin{enumerate}\itemsep=0pt
\item[(1)] $\delta$ has matrix $\left(\begin{smallmatrix} 0& 0 \\ 0 & 0\end{smallmatrix}\right)$ on a basis $e_1$, $e_2$ of $\Lambda$. Then the possible lines are $\mathbb{C}[[z]]e$ with $e\in \mathbb{C}e_1+\mathbb{C}e_2$, $e\neq 0$. The possibilities form a $\mathbb{P}^1$ over $\mathbb{C}$.
\item[(2)] $\delta$ has matrix $\left(\begin{smallmatrix} 0& 1 \\ 0& 0\end{smallmatrix}\right)$ on a basis $e_1$, $e_2$ of~$\Lambda$. Then $\mathbb{C}[[z]]e_1$ is the only possible `line'.
\end{enumerate}

This additional `line' is called a ``parabolic structure'' (see~\cite{Ina06,IIS1,IIS2}) or a ``level structure''. This def\/ines a new set ${\bf S}^+(\theta_0=0,\theta_1)$ and a new moduli problem consisting of connections on the above $\mathcal{V}$ and an invariant line in $\mathcal{V}_0\otimes
\mathbb{C}[[z]]$ (here $\mathcal{V}_0$ denotes the stalk of $\mathcal{V}$ at $z=0$). There is a f\/ine moduli space which we call $\mathcal{M}^+(\theta_0=0,\theta_1)$. The natural morphism $\mathcal{M}^+(\theta_0= 0,\theta_1)\rightarrow \mathcal{M}(\theta_0= 0,\theta_1)$ turns out to be the resolution of the latter space, seen as a surface over $\mathbb{C}(t)$ or as a~surface after f\/ixing a value for $t$. The preimage of the singular point is the projective line over~$\mathbb{C}$. This construction is also present in the papers \cite{vdP-T1,vdP-T2} and we will not make it explicit here.

Something similar has to be done for the case $\theta_0\neq 0$, $\theta_1=0$ and the case $\theta_0=\theta_1=0$. In all cases we write $\mathcal{M}^+(\theta_0,\theta_1)$ for the moduli space obtained in this way and, for notational convenience we write $\mathcal{M}^+(\theta_0,\theta_1)=\mathcal{M}(\theta_0,\theta_1)$ also in the case $\theta_0\neq 0$, $\theta_1\neq 0$.

\begin{Observation}\label{1.1} Consider a lattice $\Lambda$ over $\mathbb{C}[[z]]$ of rank two with a differential operator $\delta$ which has operator form $z\frac{d}{dz}+\left(\begin{smallmatrix} a& c\\ b& d\end{smallmatrix}\right)$ with $a,b,c,d\in \mathbb{C}[[z]]$ and $\left(\begin{smallmatrix}a & c\\ b& d \end{smallmatrix}\right)\equiv \left(\begin{smallmatrix}\alpha & 0 \\ 0 & \beta \end{smallmatrix}\right) \mod (z)$. Suppose that $\beta-\alpha \not\in \mathbb{Z}_{<0}$. Then $\Lambda$ has a unique direct summand $\mathbb{C}[[z]]e$ such that $\delta e=\alpha e$.
\end{Observation}

The proof is obtained by conjugating the operator $z\frac{d}{dz}+\left(\begin{smallmatrix}a& c\\ b& d\end{smallmatrix}\right)$ with a suitable invertible matrix $\left(\begin{smallmatrix}t_1 & 0\\ t_2& t_3 \end{smallmatrix}\right)\in {\rm GL}_2(\mathbb{C}[[z]])$. After conjugation with a~diagonal matrix in ${\rm GL}_2(\mathbb{C}[[z]])$ one can suppose that the dif\/ferential operator is $z\frac{d}{dz}+\left(\begin{smallmatrix}\alpha & c\\ b & \beta
 \end{smallmatrix}\right)$ and $b,c\in z\mathbb{C}[[z]]$. There exists $x\in \mathbb{C}[[z]]$ such that $\left(\begin{smallmatrix}1& 0\\ -x& 1
\end{smallmatrix}\right)\big\{z\frac{d}{dz}+\left(\begin{smallmatrix}\alpha & c\\ b & \beta \end{smallmatrix}\right) \big\}\left(\begin{smallmatrix}1& 0\\ x& 1 \end{smallmatrix}\right)$ equals $z\frac{d}{dz}+\left(\begin{smallmatrix}\alpha & *\\ 0 & \beta\end{smallmatrix}\right)$. Indeed, this condition is equivalent to the equation $z\frac{dx}{dz}+(\beta -\alpha)x+b-cx^2=0$. Write $x=\sum\limits_{n\geq 1}a_nz^n$. The equation reads
\begin{gather*} \sum _{n\geq 1}(n+\beta -\alpha)a_nz^n+\sum_{n\geq 1} b_nz^n
-\sum_{n\geq 1} c_nz^n\cdot \bigg(\sum _{n\geq 1}a_nz^n\bigg)^2=0 \end{gather*}
and there is a unique solution.

Now we consider the moduli space $\mathcal{M}(\theta_0,\theta_1)$ for some $\theta_0\in \mathbb{Z}_{>0}$. From Observation~\ref{1.1} one concludes that the invariant lattice $\Lambda$ over $\mathbb{C}[[z]]$ has a direct summand $\mathbb{C}[[z]]e$ such that $\delta e=-\frac{\theta_0}{2}e$. In other words, there is a unique `line' present in this situation and the `level structure' at $z=0$ is already present. This explains why $\mathcal{M}(\theta_0,\theta_1)$ has no singularities for non zero integer values of~$\theta_0$ and~$\theta_1$.

Another interesting way to produce the correct moduli space for the cases $\theta_0=0$ and/or $\theta_1=0$ is the following. A dif\/ferential module $M$ with the required data is represented by a~connection on a~vector bundle $\mathcal{V}$ of rank~2, with degree~$-3$ instead of~$-1$. Then $\mathcal{V}$ is identif\/ied with $O(-[\infty])e_1\oplus O(-2[\infty ])e_2$. The connection $\nabla \colon \mathcal{V}\rightarrow \mathcal{V}\otimes \Omega([0]+[1]+2[\infty ])$ is prescribed by the local dif\/ferential operators
\begin{gather*}
z\frac{d}{dz}+\left(\begin{matrix} \frac{\theta_0}{2} & 0\\ 0 & -\frac{\theta_0}{2}-1 \end{matrix}\right), \qquad (z-1)\frac{d}{dz}+\left(\begin{matrix}\frac{\theta_1}{2}& 0 \\ 0 & -\frac{\theta_1}{2}-1 \end{matrix}\right), \qquad z\frac{d}{dz}+\left(\begin{matrix} -\frac{3}{4}& t^2z\\ 1& -\frac{1}{4} \end{matrix}\right).
\end{gather*}

As in Section~\ref{section2.3} one has to consider two af\/f\/ine parts. On the f\/irst part the operator reads
\begin{gather*}
\frac{d}{dz}+\frac{1}{z(z-1)}\left(\begin{matrix} a_0& c_0+c_1z+c_2z^2+c_3z^3\\ z+b_0& d_0+d_1z+d_2z^2
\end{matrix}\right)
\end{gather*} and on the second part it is
\begin{gather*}
\frac{d}{dz}+\frac{1}{z(z-1)}\left(\begin{matrix} a_2z^2& c_0+c_1z+c_2z^2+c_3z^3\\ 1+b_1z & d_0+d_1z+d_2z^2 \end{matrix}\right).
\end{gather*} The equations for the entries in these matrices are similar to those of Sections~\ref{section2.3.1} and~\ref{section2.3.2}. After a computation one f\/inds that singular points only occur for the cases $\theta_0=-1$ and/or $\theta_1=-1$. In particular, this produces a smooth moduli space for, say, $\theta_0=0$ and $\theta_1\neq -1$. Other choices for~$\mathcal{V}$ with negative odd degree $-2d-1$ and local equations at $z=0$ and $z=1$ where the above matrices have traces $-d$, $-d$ can be used to construct smooth moduli spaces for all combinations of~$\theta_0$ and~$\theta_1$.

In the sequel we write $\mathcal{M}^+(\theta_0,\theta_1)$ for the resolution of the space $\mathcal{M}(\theta_0,\theta_1)$. We note that $\mathcal{M}^+(\theta_0,\theta_1)$ is in general not yet the Okamoto--Painlev\'e space for the following reason. For any $a\in \mathbb{C}^*$, the closed subspace of $\mathcal{M}^+(\theta_0,\theta_1)$, given by $t=a$, is simply connected. This follows from Observation~\ref{Observatio1.2} and the fact that a~f\/ibre of~$\mathcal{M}^+(\theta_0,\theta_1)\rightarrow \mathcal{M}(\theta_0,\theta_1)$ is either a~point or a~projective line over~$\mathbb{C}$.

The `$t$-part' of $\mathcal{M}^+(\theta_0,\theta_1)$ runs in $\mathbb{C}^*$, which is not simply connected. In the earlier def\/inition of Okamoto's space of initial values `simply connected' was required. However, in one of the later papers of Okamoto et al.~\cite{OKSO} the condition `simply connected' for `the space of initial values' is removed.

The other reason to replace $t$ by $e^{2\pi i u}$ with $u\in \mathbb{C}$ is the following. The map from a tuple $(M,\theta_0,\theta_1,t)$ to the monodromy data at $z=\infty$ depends on the choice of a direction for multisummation. This direction has to be dif\/ferent from the singular direction and the latter moves with~$t$.

As we will see in Section~\ref{section4}, the monodromy space $\mathcal{R}(s_0,s_1)$ has for the values $s_0=\pm 2$ and $s_1=\pm 2$ singular points. Further $s_0=e^{\pi i\theta_0} +e^{-\pi i\theta _0}$ and $s_1=e^{\pi i\theta_1}+e^{-\pi i \theta _1}$. Also in this case one has to add a similar level structure as a method to obtain a desingularisation $\mathcal{R}^+(s_0,s_1)$ of~$\mathcal{R}(s_0,s_1)$. Further we will show that $\mathcal{R}^+(s_0,s_1)$ is simply connected.

\section[Computation of the Painlev\'e equation ${\rm degP}_{\rm V}$]{Computation of the Painlev\'e equation $\boldsymbol{{\rm degP}_{\rm V}}$}\label{section3}
This calculation is done on the f\/irst af\/f\/ine part of the space $\mathcal{M}(\theta_0,\theta_1)$. There an isomonodromic family $\frac{d}{dz}+A$ with $A=A(z,t)$ is given and computed by the assumption that this operator commutes with an unknown operator $\frac{d}{dt}+B$ with $B=B(z,t)$. This is equivalent to the formula $\frac{\partial A}{\partial t}=\frac{\partial B}{\partial z}+[A,B]$.

From the singularities of $A$ one derives that $B$ is w.r.t.\ $z$ a polynomial matrix of degree at most~1. Further, both $A$ and $B$ are $2\times 2$-matrices with trace~0. This we use to make the computations smoother.

Write $H=\left(\begin{smallmatrix}1& 0\\ 0& -1\end{smallmatrix}\right)$, $E_1=\left(\begin{smallmatrix}0& 1\\ 0& 0\end{smallmatrix}\right)$, $E_2=\left(\begin{smallmatrix}0& 0 \\ 1& 0\end{smallmatrix}\right)$. Observe $[H,E_1]=2E_1$, $[H,E_2]=-2E_2$, $[E_1,E_2]=H$. In the sequel we write $f'$ for $\frac{df}{dt}$. Write $A=\frac{a_0}{z(z-1)}H+\frac{c_0+c_1z+t^2z^2}{z(z-1)}E_1+\frac{z+b_0}{z(z-1)}E_2$. Write $B=B_HH+B_1E_1+B_2E_2$ and $B_H=B_{H,0}+B_{H,1}z$, $B_1=B_{1,0}+B_{1,1}z$, $B_2=B_{2,0}+B_{2,1}z$ where the $B_{*,*}$ only depend on~$t$. The equation $\frac{\partial A}{\partial t}=\frac{\partial B}{\partial z}+[A,B]$, multiplied by $z(z-1)$, has coef\/f\/icients with respect to the basis $H$, $E_1$, $E_2$ which read:
\begin{alignat*}{3}
 & (H)\quad && a_0'=z(z-1)B_{H,1}+\big(c_0+c_1z+t^2z^2\big)(B_{2,0}+B_{2,1}z)-(z+b_0)(B_{1,0}+zB_{1,1}), &\\
 & (E_1)\quad && c_0'+c_1'z+2tz^2=z(z-1)B_{1,1}+2a_0(B_{1,0}+B_{1,1}z)& \\
 &&& \qquad{} -2(B_{H,0}+B_{H,1}z)\big(c_0+c_1z+t^2z^2\big), &\\
& (E_2)\quad && b_0'=z(z-1)B_{2,1}-2a_0(B_{2,0}+B_{2,1}z)+2(B_{H,0}+B_{H,1}z)(z+b_0).&
\end{alignat*}
Each of these three equations is considered with respect to the degrees in $z$. A sequence of solving equations (in a suitable order!) yields the following:
\begin{enumerate}\itemsep=0pt
\item[$(H)$] degree 3 implies $B_{2,1}=0$; degree 2 implies $0=B_{H,1}+t^2B_{2,0}-B_{1,1}$ and so $B_{2,0}=2t^{-1}$; degree 1 implies $0=c_1B_{2,0}-B_{1,0}-b_0B_{1,1}$ and so $B_{1,0}=c_1\cdot 2t^{-1}-b_0\cdot 2t$; degree 0 implies $a_0'=c_0B_{2,0}-b_0B_{1,0}$ and so $a_0'=2t^{-1}c_0-b_0(2t^{-1}c_1-2tb_0)$;
\item[$(E_1)$] degree 3 implies $B_{H,1}=0$; degree 2 implies $2t=B_{1,1}-2B_{H,0}t^2$ and so $B_{1,1}=2t$; degree~1 implies $c_1'=-B_{1,1}+2a_0B_{1,1}$ and so $c_1'=2t(2a_0-1)$; degree~0 implies $c_0'=2a_0B_{1,0}$ and so $c_0'=2a_0(c_1\cdot 2t^{-1}-b_0\cdot 2t)$;
\item[$(E_2)$] $b_0'=-2a_0B_{2,0}+2B_{H,0}(z+b_0)$ and thus $B_{H,0}=0$ and $b_0'=-2a_0B_{2,0}$ and so $b_0'=-2a_0\cdot 2t^{-1}$.
\end{enumerate}

This leads to the following set of equations:
\begin{alignat*}{3}
& (1)\quad && a_0^2+b_0c_0=\frac{\theta_0^2}{4},&\\
& (2)\quad && c_0=-\frac{\theta_0^2}{4}+\frac{\theta_1^2}{4}-c_1-t^2-b_0c_1-b_0t^2, &\\
& (3)\quad && a_0'=2t^{-1}c_0-b_0\big(2t^{-1}c_1-2tb_0\big),& \\
& (4)\quad && c_1'=-2t+4a_0t,& \\
& (5)\quad && c_0'=2a_0\big(2t^{-1}c_1-2tb_0\big),&\\
& (6)\quad && b_0'=-4a_0t^{-1}.&
\end{alignat*}
This is solved, using MAPLE, by the following steps: Eliminate $c_0$ by~(2) and $a_0$ by using~(6). In the new set of equations one can eliminate~$c_1$ in a~linear way. Then MAPLE yields a~second-order equation for~$b_0$. Write the matrix dif\/ferential operator as $\frac{d}{dz}+\left(\begin{smallmatrix} a & c\\ b& -a \end{smallmatrix}\right)$. Then taking the f\/irst vector as cyclic vector one obtains a scalar equation $\big(\frac{d}{dz}\big)^2-\frac{c'}{c}\frac{d}{dz}-a'-a^2-bc+a\frac{c'}{c}$. The~$q$ for the Painlev\'e equation is the pole $\neq 0,1,\infty $ of this
scalar equation. Thus $q=-b_0$. The Hamiltonian system for the Painlev\'e equation has variables~$q$ and~$p$ where~$p$ is the residue of the term $-a'-a^2-bc+a\frac{c'}{c}$ at $z=q$. One f\/inds in this way for ${\rm degP}_{\rm V}$ and the Hamiltonian function~$\mathcal{H}$ the formulas
\begin{gather*}
q''=\frac{1}{2}\left(\frac{1}{q}+\frac{1}{q-1}\right)(q')^2-\frac{q'}{t}+\frac{2(q-1)\theta_0^2}{qt^2}-\frac{2q\theta_1^2}{(q-1)t^2}
+8q(q-1),\\
 p=\frac{t}{4}q', \qquad \mathcal{H}=\frac{2(p^2-\frac{\theta_0^2}{4})}{tq}-\frac{2(p^2-\frac{\theta_1^2}{4})}{t(q-1)}
+2qt,
\end{gather*}
with here, {\it exceptionally}, $q'=q(1-q)\frac{\partial \mathcal{H}}{\partial p}$ and $p'=-q(1-q)\frac{\partial \mathcal{H}}{\partial q}$.

We note that the formula for ${\rm degP}_{\rm V}$ coincides with the one in~\cite{vdP-Sa}.

{\it The relation between this ${\rm degP}_{\rm V}$ and the classical degenerate ${\rm P}_{\rm V}$ is the following}. Consider solutions $q(t)$ of ${\rm degP}_{\rm V}$ which are {\it even}. Then these are written as $q(t)=Q(t^2)$ for some function~$Q(s)$. One easily computes that the second-order dif\/ferential equation for~$Q$ is
\begin{gather*} Q''=\frac{1}{2}\left(\frac{1}{Q}+\frac{1}{Q-1}\right) (Q')^2-\frac{Q'}{s}+\frac{(Q-1)\frac{\theta_o^2}{2}}{Q s^2}
-\frac{Q\frac{\theta_1^2}{2}}{(Q-1)s^2}+\frac{2Q(Q-1)}{s}.
\end{gather*}
One substitutes $Q=\frac{y}{y-1}$ and f\/inds for $y$ the second-order dif\/ferential equation
\begin{gather*}
y''=\frac{1}{2}\frac{3y-1}{y(y-1)}(y')^2-\frac{y'}{s}+\frac{y(y-1)^2\theta_1^2}{2s^2}-\frac{(y-1)^2\theta_0^2}{2s^2y}
-\frac{2y}{s}.\end{gather*}
This is the classical degenerate ${\rm P}_{\rm V}(\alpha, \beta ,\gamma, \delta)$, normalized as the following special case
${\rm P}_{\rm V}\big(\frac{\theta_1^2}{2}$, $-\frac{\theta_0^2}{2},-2,0\big)$ of ${\rm P}_{\rm V}$.

\begin{Remark} The formula for ${\rm P}_{\rm V}$ in \cite{vdP-Sa} reduces to the classical formula for ${\rm P}_{\rm V}$ as well, using the same substitution $q=\frac{y}{y-1}$.
\end{Remark}

\section[The moduli space for the analytic data $\mathcal{R}(s_0,s_1)$]{The moduli space for the analytic data $\boldsymbol{\mathcal{R}(s_0,s_1)}$}\label{section4}

We reproduce here, with a slightly dif\/ferent choice of signs, the paper~\cite{vdP-Sa}. The solution space $V$ at $z=\infty$ has a basis $e_1$, $e_2$ such that the formal monodromy is $\left(\begin{smallmatrix}0& -1\\ 1& 0\end{smallmatrix}\right)$ and the only Stokes matrix has the matrix
$\left(\begin{smallmatrix}1& 0\\ e & 1\end{smallmatrix}\right)$. The topological monodromy at $z=\infty$ is the product $M_\infty =\left(\begin{smallmatrix}-e& -1\\ 1& 0\end{smallmatrix}\right)$. By multisummation in a suitable direction we combine this with the monodromy matrices~$M_0$ and~$M_1$ for loops around $z=0$ and $z=1$. Put $s_0=e^{\pi i\theta_0}+e^{-\pi i \theta_0}$, $s_1=e^{\pi i\theta_1}+e^{-\pi i \theta_1}$. Then~$M_0$,~$M_1$ have determinants 1 and traces~$s_0$,~$s_1$. One has the relation $M_0M_1M_\infty=1$.

The basis $e_1$, $e_2$ is unique up to $e_1,e_2\mapsto \lambda e_1,\lambda e_2$. This transformation acts trivially on the matrices. Write $M_1=\left(\begin{smallmatrix}a_1& b_1 \\ c_1& d_1\end{smallmatrix}\right)$. Then $M_0$ is the inverse of $M_1M_\infty$. Hence $\mathcal{R}(s_0,s_1)$ is the af\/f\/ine space with coordinate ring $\mathbb{C}[a_1,b_1,c_1,d_1,e]$ with the relations
\begin{gather*} a_1d_1-b_1c_1=1, \qquad a_1+d_1=s_1,\qquad -a_1e+b_1-c_1=s_0.
\end{gather*}
Elimination of $c_1$, $d_1$ and $x_1=-b_1$, $x_2=-a_1$, $x_3=-e$ leads to the ring $\mathbb{C}[x_1,x_2,x_3]/(x_1x_2x_3+x_1^2+x_2^2+s_0x_1+s_1x_2+1)$.

Only for $s_1=\pm 2$ and for $s_0=\pm 2$ the corresponding af\/f\/ine cubic surface $\mathcal{R}(s_0,s_1)$ has singularities:
\begin{gather*}
s_1=\pm 2,\qquad x_1=0,\qquad x_2=\mp 1, \qquad x_3=\pm s_0, \qquad \text{and}\\
s_0=\pm 2,\qquad x_1=\mp 1, \qquad x_2=0, \qquad x_3=\pm s_1.
\end{gather*}
In particular for $s_0\neq \pm 2$, $s_1\neq \pm 2$ the space $\mathcal{R}(s_0,s_1)$ has no singularities.

\begin{Observation} $\mathcal{R}(s_0,s_1)$ is simply connected for all $s_0$, $s_1$.
 \end{Observation}
 \begin{proof} Consider the projection of $\mathcal{R}(s_0,s_1)\rightarrow \mathbb{C}^2$, by $(x_1,x_2,x_3)\mapsto (x_1,x_2)$. The f\/ibres of this map are either a~point or $\mathbb{C}$ or empty. The image $B$ is the union of~$(\mathbb{C}^*)^2$ with the points $\{(0,x_2)\,|\, x_2^2+s_1x_2+1=0\}$ and $\{(x_1,0)\,|\, x_1^2+s_0x_1+1=0\}$. It suf\/f\/ices to show that~$B$ is simply connected. The inclusion $(\mathbb{C}^*)^2\subset B$ induces a surjection
$\pi_1((\mathbb{C}^*)^2,*)\rightarrow \pi_1(B,*)$. Consider a generator of $\pi_1((\mathbb{C}^*)^2,*)$, represented by the loop $s\in [0,1]\mapsto (\tilde{x}_1,e^{2\pi i s})$, where $\tilde{x}_1$ is chosen such that $\tilde{x}_1^2+s_0\tilde{x}_1+1=0$. In $B$ this loop is homotopic to
the constant loop $s\mapsto (\tilde{x}_1,0)$ by the homotopy $(s,\lambda) \in [0,1]^2\mapsto (\tilde{x}_1,\lambda e^{2\pi i s})$. One concludes that the two generators of $\pi_1((\mathbb{C}^*)^2,*)$ have trivial image in $\pi_1(B,*)$ and that $\pi_1(B,*)=1$.
 \end{proof}

As in Section~\ref{section2.3.3}, {\it the geometric way} to resolve the singularities of $\mathcal{R}(s_0,s_1)$ for $s_0=\pm 2$ and/or $s_1=\pm 2$ is to add a level structure consisting of a line (or two lines if both $s_0=\pm 2$ and $s_1=\pm 2$). The resulting space is denoted by $\mathcal{R}^+(s_0,s_1)$. We will work out the details for $s_0\neq \pm 2$ and $s_1=2$.

The f\/ibre of the surjective morphism $\mathcal{R}^+(s_0,s_1)\rightarrow \mathcal{R}(s_0,s_1)$ is in general a point and there are at most two f\/ibres isomorphic to $\mathbb{P}^1$. It follows that {\it $\mathcal{R}^+(s_0,s_1)$ is simply connected as well}.

For the formulation of the (extended) Riemann--Hilbert morphism we need to replace the space $T=\mathbb{C}^*$ of the variable $t$ by its universal covering $\tilde{T}=\mathbb{C}$. The reason is that the singular direction at inf\/inity varies with~$t$. Then
\begin{gather*} {\rm RH}^+\colon \ \mathcal{M}^+(\theta_0,\theta_1)\times _T\tilde{T}\rightarrow \mathcal{R}^+(s_0,s_1)\times \tilde{T}\end{gather*}
is a well def\/ined analytic map. This (extended) Riemann--Hilbert map ${\rm RH}^+$ is bijective on points. Indeed, the points on the left hand side correspond to the tuples $(M,\theta_0,\theta_1,t)$ with additionally $u\in \mathbb{C}$ with $t=e^{2\pi i u}$ and level structure(s) if needed. The points on the right hand side correspond to the analytic data with level structure (if needed), $u\in \mathbb{C}$ and the formal structure at $z=\infty$. By \cite[Theorem~1.7]{vdP-Sa}, these two sets coincide. As in \cite[Theorem~1.5]{vdP1} we conclude that ${\rm RH}^+$ is an analytic isomorphism between two algebraic varieties over~$\mathbb{C}$. Moreover, as in the proof of {\it loc. sit.}, from the isomorphism one obtains (compare~\cite{OKSO})

\begin{Theorem}
The Painlev\'e property for ${\rm degP}_{\rm V}(\theta_0,\theta_1)$ holds. Moreover $\mathcal{M}^+(\theta_0,\theta_1)\times_T \tilde{T}$ is the
Okamoto--Painlev\'e space.
\end{Theorem}

{\it The resolution $\mathcal{R}^+(s_0,2)\rightarrow \mathcal{R}(s_0,2)$ for $s_0\neq \pm 2$}. As before, any dif\/ferential module $M$ (with the given data) determines a basis $e_1$, $e_2$ of the solution space $V$ at $z=\infty$ such that the formal monodromy $\gamma$ has matrix $\left(\begin{smallmatrix} 0& -1\\ 1& 0\end{smallmatrix}\right)$. This basis is unique up to multiplication of $e_1$, $e_2$ by the same constant. All maps are written as matrices with respect to this basis.

A point of $\mathcal{R}^+(s_0,2)$ corresponds to a tuple $(M_0,M_1,M_\infty,\mathbb{C}v)$ where $M_0$, $M_1$, $M_\infty$ are the matrices for the topological monodromies for the points $0$, $1$, $\infty$. Then $M_0M_1M_\infty =1$ and $v$ is a non zero eigenvector of~$M_1$ (for the
eigenvalue~1). Write $M_1=\left(\begin{smallmatrix}a_1& b_1\\ c_1& d_1 \end{smallmatrix}\right)$, $M_\infty=\left(\begin{smallmatrix} -e& -1\\ 1& 0 \end{smallmatrix}\right)$, $v=\left(\begin{smallmatrix} y_0\\ y_1\end{smallmatrix}\right)$. There are two af\/f\/ine parts given by the cases $y_0\neq 0$ and $y_1\neq 0$. We consider here the f\/irst case and normalize $y_0=1$.

The matrix $M_0$ is determined by $M_0M_1M_\infty=1$. The equations in the variables $a_1$, $b_1$, $c_1$, $d_1$, $e$, $y_1$ are
\begin{gather*} a_1+d_1=2,\qquad a_1d_1-b_1c_1=1,\qquad a_1-1+b_1y_1=0, \qquad c_1+(d_1-1)y_1=0, \end{gather*}
where the last two equations come from: $v$ is eigenvector for $M_1$. One eliminates $c_1$, $d_1$ by $a_1+d_1=2$ and $-a_1e+b_1-c_1=s_0$. One writes $x_1=-b_1$, $x_2=-a_1$, $x_3=-e$. Thus the ring of regular functions on the af\/f\/ine part that we are looking at has the form $\mathbb{C}[x_1,x_2,x_3,y_1]/{\rm relations}$ and the relations are
\begin{gather*}
x_1x_2x_3+x_1^2+x_2^2+s_0x_1+2x_2+1=0, \qquad (x_2+1)+x_1y_1=0,\\
 (1+x_2)y_1=x_2x_3+x_1+s_0.
 \end{gather*}
The morphism $\mathcal{R}^+(s_0,2)\rightarrow \mathcal{R}(s_0,2)$, restricted to this af\/f\/ine part, is given by the obvious homomorphism
\begin{gather*}\mathbb{C}[x_1,x_2,x_3]/\big( x_1x_2x_3+x_1^2+x_2^2+s_0x_1+2x_2+1\big)
\rightarrow \mathbb{C}[x_1,x_2,x_3,y_1]/{\rm relations}.\end{gather*}
For $x_1\neq 0$ one can eliminate $y_1$ by using the equation $(x_2+1)+x_1y_1=0$. In fact, the above map is an isomorphism after inverting the element~$x_1$. For $x_1=0$ one f\/inds that $(x_1,x_2,x_3)=(0,-1,s_0)$ and this is the unique singular point of $\mathcal{R}(s_0,2)$. The points lying above this singular point are $(x_1,x_2,x_3,y_1)=(0,-1,s_0,a)$ for all $a\in \mathbb{C}$. One easily verif\/ies that these points of $\mathcal{R}^+(s_0,2)$ are smooth.

The computation of the af\/f\/ine part $y_1\neq 0$ is similar. One concludes that $\mathcal{R}^+(s_0,2)\rightarrow \mathcal{R}(s_0,2)$ is a resolution of singularities and that the f\/ibre above the singular point is $\mathbb{P}^1$.

\section{Computing the B\"acklund transformations}

\looseness=1 Write $\bf S$ for the union of the sets ${\bf S}(\theta_0,\theta_1)$ taken over all $\theta_0$, $\theta_1$. We recall that $\theta_0$, $\theta_1$, $t$ determine the invariant lattices at $z=0,1,\infty$. A `natural' automorphism of~$\bf S$ may change a~given tuple $(M,\theta_0,\theta_1,t)$ into the same module $M$ but with dif\/ferent lattices. For example, the lattice at $z=0$ will be changed by replacing $\theta_0$ by $\theta_0+2$. Further the module~$M$ can be changed and one can consider the automorphism of $\mathbb{P}^1$ which interchanges $z=0,1$ and has $z=\infty$ as f\/ixed point.

 Some `natural' automorphisms of $\bf S$ are given by the tuple $(M,\theta_0,\theta_1,t) \mapsto $
 \begin{enumerate}\itemsep=0pt
\item[(1)] $(M,\theta_0,\theta_1,-t)$,
\item[(2)] $(M,-\theta_0,\theta_1,t)$,
\item[(3)] $(M,\theta_0,-\theta_1,t)$,
\item[(4)] $(\tilde{M},\theta_1,\theta_0)$ where $\tilde{M}$ is obtained from $M$ by the automorphism $z\mapsto 1-z$,
\item[(5)] $(N\otimes M, \theta _0+1,\theta_1,t)$ where $N$ is the 1-dimensional module represented by $\frac{d}{dz}+\frac{1}{2z}$.
\end{enumerate}
These special automorphisms will be lifted to isomorphisms between various moduli spaces $\mathcal{M}^+(\theta_0,\theta_1)$. The isomorphisms preserve the analytic data and map therefore solutions of one ${\rm degP}_{\rm V}$ equation to solutions of another. They are B\"acklund transformations for the ${\rm degP}_{\rm V}$ and we hope that we found all of them in this way.

{\it Discussion of the transformations}: (1) only changes the variable. It induces the B\"acklund transformation which sends a solution $q(t)$ of
 ${\rm degP}_{\rm V}(\theta_0,\theta_1)$ to another solution $q(-t)$ of that equation. (2)~and~(3) induce the identity on the solutions of all ${\rm degP}_{\rm V}$.

(4) Consider a point $M$ of $\mathcal{M}(\theta_0,\theta_1)$ belonging to the f\/irst af\/f\/ine chart of this variety.
It is represented by $\frac{d}{dz}+\frac{1}{z(z-1)}\left(\begin{smallmatrix} a_0&c_0+c_1z+t^2z^2\\ z-q& -a_0\end{smallmatrix}\right)$.
The change $z\mapsto 1-z$ applied to this operator yields
 \begin{gather*} \frac{d}{dz}+\frac{1}{z(z-1)}\left(\begin{matrix} -a_0& -c_0-c_1(1-z)-t^2(1-z)^2\\ z+q-1& a_0\end{matrix}\right).
 \end{gather*}
A small computation shows that after changing $t$ into $it$ this becomes a dif\/ferential operator belonging to $\mathcal{M}(\theta_1,\theta_0)$. This produces the B\"acklund transformation $q(t)\mapsto -q(it)+1$ which sends a solution of ${\rm degP}_{\rm V}(\theta_0,\theta_1)$ to a~solution of ${\rm degP}_{\rm V}(\theta_1,\theta_0)$.

(5) $M$ is locally represented by $z\frac{d}{dz}+\left(\begin{smallmatrix}\omega _0& 0\\ 0 & -\omega_0 \end{smallmatrix}\right)$ with $\omega _0 =\frac{\theta_0}{2}$ and by $z\frac{d}{dz}+\left(\begin{smallmatrix} \frac{-3}{4}& t^2z\\ 1& -\frac{1}{4} \end{smallmatrix}\right)$. After taking the tensor product with $N$ these operators become $z\frac{d}{dz}+\left(\begin{smallmatrix} \omega _0+\frac{1}{2}& 0\\ 0& -\omega_0+\frac{1}{2}\end{smallmatrix}\right)$ and $z\frac{d}{dz}+\left(\begin{smallmatrix} \frac{-3}{4}+\frac{1}{2}& t^2z\\ 1& -\frac{1}{4}+\frac{1}{2}\end{smallmatrix}\right)$. By multiplying one basis vector by $z^{-1}$, the f\/irst operator becomes $z\frac{d}{dz}+\left(\begin{smallmatrix} \omega _0+\frac{1}{2}& 0\\ 0& -\omega_0-\frac{1}{2}\end{smallmatrix}\right)$ and the second becomes $z\frac{d}{dz}+\left(\begin{smallmatrix} \frac{-3}{4}-\frac{1}{2}& t^2z\\ 1& -\frac{1}{4}+\frac{1}{2}\end{smallmatrix}\right)$. This shows the validity of~(5).

Using MAPLE one can compute the actual isomorphism between the two moduli spaces. In terms of matrix dif\/ferential operators, this works as follows. Let the dif\/ferential operator $\frac{d}{dz}+A(a,q)$ represent an open part of the f\/irst af\/f\/ine chart $\mathcal{M}_1(\theta_0,\theta_1)$. Here $q$ stands for $-b_0$ (as before) and~$a$ denotes~$a_0$. Further $c_1$ is written as a~rational function in~$a$ and~$q$.

Similarly, let $\frac{d}{dz}+\tilde{A}(\tilde{a},\tilde{q})$ denote the dif\/ferential operator on an open part of $\mathcal{M}(\theta_0+1,\theta_1)$. Then there exists $U\in {\rm GL}_2(\mathbb{C}(z))$ such that
\begin{gather*}
U\left(\frac{d}{dz}+A(a,q)+\left(\begin{matrix} \frac{1}{2z}& 0 \\ 0 & \frac{1}{2z} \end{matrix}\right) \right)U^{-1}=\frac{d}{dz}+\tilde{A}(\tilde{a},\tilde{q}).
\end{gather*}
Further $U$, $U^{-1}$ are seen to have poles of order $\leq 1$ at $z=0$ and $z=\infty$ and no further poles. This information suf\/f\/ices for the computation of the solution~$\tilde{q}$ of ${\rm degP}_{\rm V}(\theta_0+1,\theta_1)$ in terms of the solution~$q$ of ${\rm degP}_{\rm V}(\theta_0,\theta_1)$ and its f\/irst derivative (and $t$, $\theta_0$, $\theta_1$)
\begin{gather*}
\tilde{q}=1-\frac{\theta_0^2(q-1)}{4q^2t^2}+\frac{a\theta_0}{q^2t^2}+\frac{\theta_1^2}{4t^2(q-1)}
\end{gather*}
and a longer formula for $\tilde{a}$, namely
\begin{gather*} \frac{q-1}{8q^3t^2}\theta_0^3+\frac{q^2+2aq-q-6a}{8q^3t^2} \theta_0^2+\frac{q-1}{2q}\theta_0-\frac{a(q^2+2aq-q-3a)}{2(q-1)q^3t^2}\theta_0
-\frac{\theta_0\theta_1^2}{8qt^2(q-1)}\\
\qquad{} -\frac{a}{q} - \frac{q+2a}{8q(q-1)t^2}\theta_1^2 +\frac{a^2(2a+q)}{q^3t^2(q-1)}.
\end{gather*}

{\it Now we compare the group of the B\"acklund transformations for ${\rm degP}_{\rm V}$ with the work of N.S.~Witte} \cite{Wi}. We restrict our transformations to the case of even solutions of ${\rm degP}_{\rm V}$ and f\/ind the classical ${\rm P}_{\rm V}\big(\frac{\theta_1^2}{2}, -\frac{\theta_0^2}{2},-2,0\big)$. The formulas (21), (22), (23) of~\cite{Wi} lead to $v_1=\theta_0+\theta_1$ and $v_2=\theta_0-\theta_1$. Our group of transformations for the even solutions of ${\rm degP}_{\rm V}$ is generated by $(\theta_0,\theta_1)\mapsto (\theta_0+1,\theta_1)$ and $(\theta_0,\theta_1)\mapsto (\theta_1,\theta_0)$ and the `trivial' transformations $(\theta_0,\theta_1)\rightarrow (\pm \theta_0,\pm \theta_1)$. One easily verif\/ies that this coincides with the group of transformations in~\cite{Wi}.

We remark that the transformations $\theta_0\mapsto -\theta_0$ and/or $\theta_1\mapsto -\theta_1$ act trivially on solutions of ${\rm degP}_{\rm V}$. However (compare~\cite{Wi}), they do not act as the identity on a~suitable Hamiltonian system for ${\rm degP}_{\rm V}$.

\subsection*{Acknowledgments}
The authors thank Yousuke Ohyama for his many helpful answers to our questions, and the referees for their careful reading and useful suggestions. The f\/irst author thanks the Johann Bernoulli Institute of the University of Groningen and the Universidad Sim\'{o}n Bol\'{\i}var for f\/inancial support to participate in this project.

\pdfbookmark[1]{References}{ref}
\LastPageEnding


\begin{thebibliography}{99}
\footnotesize\itemsep=0pt
\bibitem{CMR}
Chekhov L., Mazzocco M., Rubtsov V., Painlev\'e monodromy manifolds, decorated
 character varieties and cluster algebras, \href{https://doi.org/10.1093/imrn/rnw219}{\textit{Int. Math. Res. Not.}}, {t}o
 appear, \href{http://arxiv.org/abs/1511.03851}{arXiv:1511.03851}.

\bibitem{FIKN}
Fokas A.S., Its A.R., Kapaev A.A., Novokshenov V.Yu., Painlev\'e transcendents.
 {T}he {R}iemann--{H}ilbert approach, \href{https://doi.org/10.1090/surv/128}{\textit{Mathematical Surveys and
 Monographs}}, Vol.~128, Amer. Math. Soc., Providence, RI, 2006.

\bibitem{Gr1}
Gromak V.I., On the theory of {P}ainlev\'e's equations, \textit{Differential Equations} \textbf{11} (1975), 285--287.

\bibitem{Ina06}
Inaba M., Moduli of parabolic connections on curves and the
 {R}iemann--{H}ilbert correspondence, \href{https://doi.org/10.1090/S1056-3911-2013-00621-9}{\textit{J.~Algebraic Geom.}} \textbf{22}
 (2013), 407--480, \href{http://arxiv.org/abs/math.AG/0602004}{math.AG/0602004}.

\bibitem{IISA}
Inaba M., Iwasaki K., Saito M.-H., Dynamics of the sixth {P}ainlev\'e
 equation, in Th\'eories asymptotiques et \'equations de {P}ainlev\'e,
 \textit{S\'emin. Congr.}, Vol.~14, Soc. Math. France, Paris, 2006, 103--167.

\bibitem{IIS1}
Inaba M., Iwasaki K., Saito M.-H., Moduli of stable parabolic connections,
 {R}iemann--{H}ilbert correspondence and geometry of {P}ainlev\'e equation of
 type~{VI}.~{I}, \textit{Publ. Res. Inst. Math. Sci.} \textbf{42} (2006),
 987--1089, \href{http://arxiv.org/abs/math.AG/0309342}{math.AG/0309342}.

\bibitem{IIS2}
Inaba M., Iwasaki K., Saito M.-H., Moduli of stable parabolic connections,
 {R}iemann--{H}ilbert correspondence and geometry of {P}ainlev\'e equation of
 type~{VI}.~{II}, in Moduli Spaces and Arithmetic Geometry, \textit{Adv. Stud.
 Pure Math.}, Vol.~45, Math. Soc. Japan, Tokyo, 2006, 387--432,
 \href{http://arxiv.org/abs/math.AG/0605025}{math.AG/0605025}.

\bibitem{JMU}
Jimbo M., Miwa T., Ueno K., Monodromy preserving deformation of linear ordinary
 dif\/ferential equations with rational coef\/f\/icients. {I}.~{G}eneral theory and
 {$\tau $}-function, \href{https://doi.org/10.1016/0167-2789(81)90013-0}{\textit{Phys.~D}} \textbf{2} (1981), 306--352.

\bibitem{JM}
Jimbo M., Miwa T., Monodromy preserving deformation of linear ordinary
 dif\/ferential equations with rational coef\/f\/icients.~{II}, \href{https://doi.org/10.1016/0167-2789(81)90021-X}{\textit{Phys.~D}}
 \textbf{2} (1981), 407--448.

\bibitem{OKSO}
Ohyama Y., Kawamuko H., Sakai H., Okamoto K., Studies on the {P}ainlev\'e
 equations. {V}.~{T}hird {P}ainlev\'e equations of special type {${\rm P}_{\rm
 III}(D_7)$} and {${\rm P}_{\rm III}(D_8)$}, \textit{J.~Math. Sci. Univ.
 Tokyo} \textbf{13} (2006), 145--204.

\bibitem{OO2}
Ohyama Y., Okumura S., R.~{F}uchs' problem of the {P}ainlev\'e equations from
 the f\/irst to the f\/ifth, in Algebraic and Geometric Aspects of Integrable
 Systems and Random Matrices, \href{https://doi.org/10.1090/conm/593/11876}{\textit{Contemp. Math.}}, Vol.~593, Amer. Math.
 Soc., Providence, RI, 2013, 163--178, \href{http://arxiv.org/abs/math.CA/0512243}{math.CA/0512243}.

\bibitem{O1}
Okamoto K., Sur les feuilletages associ\'es aux \'equations du second ordre \`a
 points critiques f\/ixes de {P}.~{P}ainlev\'e. Espaces des conditions
 initiales, \textit{Japan.~J. Math.} \textbf{5} (1979), 1--79.

\bibitem{O2}
Okamoto K., Isomonodromic deformation and {P}ainlev\'e equations, and the
 {G}arnier system, \textit{J.~Fac. Sci. Univ. Tokyo Sect. IA Math.}
 \textbf{33} (1986), 575--618.

\bibitem{O3}
Okamoto K., Studies on the {P}ainlev\'e equations. {III}.~{S}econd and fourth
 {P}ainlev\'e equations, {${\rm P}_{{\rm II}}$} and {${\rm P}_{{\rm IV}}$},
 \href{https://doi.org/10.1007/BF01458459}{\textit{Math. Ann.}} \textbf{275} (1986), 221--255.

\bibitem{O4}
Okamoto K., Studies on the {P}ainlev\'e equations. {IV}.~{T}hird {P}ainlev\'e
 equation {${\rm P}_{{\rm III}}$}, \textit{Funkcial. Ekvac.} \textbf{30}
 (1987), 305--332.

\bibitem{O5}
Okamoto K., The {H}amiltonians associated to the {P}ainlev\'e equations, in The
 {P}ainlev\'e property, \href{https://doi.org/10.1007/978-1-4612-1532-5_13}{\textit{CRM Ser. Math. Phys.}}, Springer, New York, 1999,
 735--787.

\bibitem{vdP2}
van~der Put M., Families of linear dif\/ferential equations related to the second
 {P}ainlev\'e equation, in Algebraic Methods in Dynamical Systems,
 \href{https://doi.org/10.4064/bc94-0-18}{\textit{Banach Center Publ.}}, Vol.~94, Polish Acad. Sci. Inst. Math., Warsaw,
 2011, 247--262.

\bibitem{vdP1}
van~der Put M., Families of linear dif\/ferential equations and the {P}ainlev\'e
 equations, in Geometric and Dif\/ferential {G}alois Theories, \textit{S\'emin.
 Congr.}, Vol.~27, Soc. Math. France, Paris, 2013, 207--224.

\bibitem{vdP-Sa}
van~der Put M., Saito M.-H., Moduli spaces for linear dif\/ferential equations and
 the {P}ainlev\'e equations, \href{https://doi.org/10.5802/aif.2502}{\textit{Ann. Inst. Fourier (Grenoble)}}
 \textbf{59} (2009), 2611--2667, \href{http://arxiv.org/abs/0902.1702}{arXiv:0902.1702}.

\bibitem{vdP-Si}
van~der Put M., Singer M.F., Galois theory of linear dif\/ferential equations,
 \href{https://doi.org/10.1007/978-3-642-55750-7}{\textit{Grundlehren der Mathematischen Wissenschaften}}, Vol.~328,
 Springer-Verlag, Berlin, 2003.

\bibitem{vdP-T1}
van~der Put M., Top J., A {R}iemann--{H}ilbert approach to {P}ainlev\'e~{IV},
 \href{https://doi.org/10.1080/14029251.2013.862442}{\textit{J.~Nonlinear Math. Phys.}} \textbf{20} (2013), suppl.~1, 165--177,
 \href{http://arxiv.org/abs/1207.4335}{arXiv:1207.4335}.

\bibitem{vdP-T2}
van~der Put M., Top J., Geometric aspects of the {P}ainlev\'e equations {${\rm
 PIII}(\rm D_6)$} and {${\rm PIII}(\rm D_7)$}, \href{https://doi.org/10.3842/SIGMA.2014.050}{\textit{SIGMA}} \textbf{10}
 (2014), 050, 24~pages, \href{http://arxiv.org/abs/1207.4023}{arXiv:1207.4023}.

\bibitem{S-Ta}
Saito M.-H., Takebe T., Classif\/ication of {O}kamoto--{P}ainlev\'e pairs,
 \textit{Kobe~J. Math.} \textbf{19} (2002), 21--50, \href{http://arxiv.org/abs/math.AG/0006028}{math.AG/0006028}.

\bibitem{STT02}
Saito M.-H., Takebe T., Terajima H., Deformation of {O}kamoto--{P}ainlev\'e
 pairs and {P}ainlev\'e equations, \href{https://doi.org/10.1090/S1056-3911-01-00316-2}{\textit{J.~Algebraic Geom.}} \textbf{11}
 (2002), 311--362, \href{http://arxiv.org/abs/math.AG/0006026}{math.AG/0006026}.

\bibitem{STe02}
Saito M.-H., Terajima H., Nodal curves and {R}iccati solutions of {P}ainlev\'e
 equations, \href{https://doi.org/10.1215/kjm/1250283083}{\textit{J.~Math. Kyoto Univ.}} \textbf{44} (2004), 529--568,
 \href{http://arxiv.org/abs/math.AG/0201225}{math.AG/0201225}.

\bibitem{Wi}
Witte N.S., New transformations for {P}ainlev\'e's third transcendent,
 \href{https://doi.org/10.1090/S0002-9939-04-07087-X}{\textit{Proc. Amer. Math. Soc.}} \textbf{132} (2004), 1649--1658,
 \href{http://arxiv.org/abs/math.CA/0210019}{math.CA/0210019}.

\end{thebibliography}
\end{document}